\documentclass[a4paper,12pt]{article}
\usepackage{citesort,amsmath,amsthm,amssymb,threeparttable,epsfig}
\newtheorem{assumption}{\bf Assumption}[section]
\newtheorem{lemma}{\bf Lemma}[section]
\newtheorem{proposition}{\bf Proposition}[section]
\newtheorem{theorem}{\bf Theorem}[section]
\newtheorem{remark}{\bf Remark}[section]
\newcommand{\pair}[2]{\mbox{$\left\langle #1,#2 \right\rangle$}}

\renewcommand{\labelenumi}

\title{Convergence analysis of variants of 
the averaged alternating modified reflections method
\thanks{This work  was supported in part by the Ministry of Education, Culture, Sports, Science, and Technology [grant number 16K05280] }}
\author{Shin-ya Matsushita\thanks{Department of Electronics and Information Systems, 
Akita Prefectural University, 84-4 Yuri-Honjo, Akita, Japan 
({\tt matsushita@akita-pu.ac.jp})} }

\begin{document}

\maketitle 

\begin{abstract}
This paper presents new variants of the averaged alternating modified reflections (AAMR) 
method for the best approximation problem. 
Under a mild constraint qualification, we first show 
its weak convergence and then establish a convergence rate. 
Furthermore, under a standard interior-point-like condition, we show 
that the method has a finite termination property. 
\end{abstract}


\noindent {\bf Keywords:} averaged alternating modified reflections method, 
best approximation problem, weak convergence, rate of convergence, finite termination, Hilbert space\\

\noindent {\bf MSC2010: }  47H09, 47J25, 47N10, 90C25 

\pagestyle{plain}
\thispagestyle{plain}

\section{Introduction}
Let $A$ and $B$ be closed convex subsets  of a real Hilbert space $H$. 
We consider the problem of finding the closest point from a given point $x_{0}$ in $H$ to $A\cap B$, i.e.,
\begin{equation}
\label{BAP}
\mbox{minimize}\quad \Vert u-x_{0}\Vert~~\mbox{subject to}~~ u\in A\cap B. 
\end{equation}
Problem (\ref{BAP}) is called the 
best approximation problem with respect to $A\cap B$ 
and this problem is of considerable importance in
data analysis and modeling, control system design and signal processing 
\cite{Skelton-Iwasaki-Grigoriadis,Ben-Nemirovski,Bauschke-Combettes,Deutsch,Escalante-Raydan,Tanaka-Nakata,Tong-Guo-Tong-Xi-Yu}. 
In the case when $A$ is the set of $N\times N$ 
symmetric positive semidefinite matrices $\mathcal{S}^{N}_{+}$ 
and $B$  is an appropriate subset of the set of $N\times N$ 
symmetric matrices $\mathcal{S}^{N}$ respectively, several type matrix approximation problems 
can be described as (\ref{BAP}) on space $\mathcal{S}^{N}$ 
(see, for instance, patterned covariance matrix problems \cite[Chapter 6]{Escalante-Raydan}, 
controller design problems \cite[Chapter 10]{Skelton-Iwasaki-Grigoriadis} and 
well-conditioned positive definite matrix approximation problems \cite{Tanaka-Nakata,Tong-Guo-Tong-Xi-Yu}). 


The method discussed in this paper is the {\it averaged alternating modified reflections} (AAMR) method. 
The AAMR method was introduced by Arag\'on Artacho and Campoy  \cite{Aragon-Campoy} 
to solve the best approximation problem with respect to  convex feasibility problems. 
The framework of the method for closed convex sets $A$ and $B$ is as follows: Given $x_0\in H$ and $q\in H$,
\begin{equation}
\label{AAMR1}
x_{n+1}=T_{A-q,B-q,\alpha,\beta}(x_{n}), ~~n=0,1,2,\dots, 
\end{equation}
where $\alpha, \beta\in (0,1)$, $T_{A-q,B-q,\alpha,\beta}:H\rightarrow H$ is the averaged alternating modified 
reflections operator defined by 
$T_{A-q,B-q,\alpha,\beta}=(1-\alpha)I+\alpha(2\beta P_{B-q}-I)(2\beta P_{A-q}-I)$, 
$I$ denotes the identity mapping, $C+p$ denotes a set $C$ shifted by a point $p$, i.e., 
$C+p=\{c+p:c\in C\}$ and $P_{C}$ denotes the metric projection onto $C$. 
If $A\cap B\neq\emptyset$, under the constraint qualification 
\begin{equation}
\label{CQ1}
 q-P_{A\cap B}(q)\in (N_{A}+N_{B})(P_{A\cap B}(q)), 
\end{equation}
where $N_{A}$ and $N_{B}$ denote the normal cones to the sets $A$ and $B$, respectively, 
Arag\'on Artacho and Campoy \cite[Theorem 4.1]{Aragon-Campoy} showed that 
the sequence generated by (\ref{AAMR1}) weakly converges to a point $x^{*}\in H$, such that  
\begin{equation}
\label{BAP2}
P_{A}(x^{*}+q)=P_{A\cap B}(q).
\end{equation}
That is, $P_{A}(x^{*}+q)$ solves the problem (1) when $q=x_{0}$.

Assume that $R_{A-q,B-q,\beta}=(2\beta P_{B-q}-I)(2\beta P_{A-q}-I)$ in order
to simplify the notation. 
Since $R_{A-q,B-q,\beta}$ is nonexpansive (see \cite[Proposition 3.3]{Aragon-Campoy}), 
(\ref{AAMR1}) can be viewed as the Krasnosel'ski\u\i-Mann 
fixed point iteration with respect to $R_{A-q,B-q,\beta}$, and 
this method generates weakly 
convergent iteration sequences (see, e.g., \cite[Subchapter 5.2]{Bauschke-Combettes}). 
Moreover, the weak cluster points of these weakly 
convergent iteration sequences only solve the following fixed point equation $R_{A-q,B-q,\beta}(u)=u$. 
However, it is not guaranteed whether the weak cluster points solve the best approximation problem.

The goal of this paper is three-fold. First, we show an enhanced weak convergence result for a variant of 
the AAMR method. 
Second, we establish its convergence  rate. 
The third purpose is to analyze the finite termination property. 

To describe our goal more concretely, 
we introduce the following variant of the AAMR method  for solving (\ref{BAP}):
\begin{equation}
\label{AAMR2}
y_{n}=P_{A}(x_{n}+q), ~~n=0,1,2,\dots, 
\end{equation}
where $\{x_{n}\}$ is the sequence generated by (\ref{AAMR1}). 
As we have mentioned, sequences generated by the AAMR method 
(\ref{AAMR1}) are weakly convergent. But 
it is not clear whether the sequence $\{y_n\}$ generated by (\ref{AAMR2}) weakly 
converges to $P_{A}(x^{*}+q)$  since $P_{A}$ is in general not sequentially weakly continuous \cite{Zarantonello}. 
By using the demiclosedness principle in \cite{Bauschke}, 
we show that 
$\{y_{n}\}$ weakly converges to $P_{A}(x^{*}+q)$, without any other restrictions. 

Our second purpose is to analyze the convergence rate for (\ref{AAMR2}). 
To establish the convergence rate, 
we thus will use the following residual function 
\begin{equation}
\label{Estimate}
r(x)=\Vert P_{A}(x+q)-P_{B}(P_{A}(x+q))\Vert
\end{equation}
as a measure of the convergence rate. 
Clearly, if $r(x_{n})=0$ then 
$y_{n}=P_{B}(y_{n})$, so $y_{n}$ is in $A\cap B$ because $y_{n}$ is in $A$ for all $n\in \mathbb{N}$. 
On the other hand, if $r(x_{n})$ is  large, 
then $y_{n}$ is to be far away from the set $B$. 
Therefore, the quantity $r(x_{n})$ can be viewed as a measure of 
the distance between the iteration $y_{n}$ and the set $B$. 
Recently, a comprehensive convergence rate 
analysis for operator splitting methods was studied in \cite{Deng-Lai-Peng-Yin}. 
Using a useful technique established in \cite[Lemma 2.1]{Deng-Lai-Peng-Yin}, we show that $r(x_{n})=o\left(\frac{1}{\sqrt{n}}\right)$, where the notation $o$ means 
that $s_{n}=o\left(\frac{1}{t_{n}}\right)$ if and only if $\lim_{n\rightarrow\infty}s_{n}t_{n}=0$.

Our third purpose is to analyze the finite termination property of a variant of (\ref{AAMR2}). 
Recently, under a standard interior-point-like condition ($A\cap{\rm int}B\neq\emptyset$),  
finite termination of projection-type iterative methods was studied in \cite{Rami-Helmke-Moore,B-D-N-P,M-X}. 
Using the techniques developed in \cite{Rami-Helmke-Moore,M-X}, we show that a variant of (\ref{AAMR2}) terminates 
finitely to a point in $A\cap {\rm int}B$.

The rest of this paper is organized as follows. In section 2, some preliminaries are 
presented. In section 3, we discuss the weak convergence of (\ref{AAMR2}). 
Then, we discuss the convergence rate of (\ref{AAMR2}) in section 4.  
Moreover, we investigate the finite termination in section 5. 
Finally, we make some conclusions in section 6.

\section{Basic definitions and preliminaries}
The following notations will be used in this paper: 
$\mathbb{R}$ denotes the set of real numbers; 
$\mathbb{N}$ denotes the set of nonnegative integers; 
$H$ denotes a real Hilbert space; for any $x,y\in H$, $\pair{x}{y}$ denotes the inner product of $x$ and $y$; 
for any $z\in H$, $\Vert z\Vert$ denotes the norm of $z$, i.e., $\Vert z\Vert=\sqrt{\pair{z}{z}}$; for any $\{x_{n}\}\subset H$,  $x_{n}\rightharpoonup x$ 
denotes weak convergence, i.e., $\pair{x_{n}}{x^{*}}\rightarrow\pair{x}{x^{*}}~(n\rightarrow\infty)~(\forall x^{*}\in H)$; 
for any $w\in H$ and $A\subset H$, $A+w$ denotes $A$ shifted by $w$, i.e., $A+w=\{a+w:a\in A\}$;  
for any $r>0$, $B(x,r)$ denotes a closed ball 
with center $x$ and radius $r$, i.e., $B(x,r)=\{v\in H:\Vert x-v\Vert \le r\}$;  
\mbox{int}$A$ denotes the interior of set $A$; $A^{c}$ denotes the complement of $A$; 
for any $A,B\subset H$, $\mbox{\rm dist}(A,B)$ denotes the distance between two sets $A$ and $B$, i.e.,
$ \mbox{\rm dist}(A,B)=\inf\{\Vert x-y\Vert:x\in A, y\in B\}$;  
for any $C\subset H$ and mapping $U:C\rightarrow C$, $\mbox{Fix}(U)$ 
denotes the fixed point set of $U$, i.e., $\mbox{Fix}(U)=\{x\in C: U(x)=x\}$. 

Let $C$ be a closed and convex subset of $H$. A mapping $U:C\rightarrow H$ is said to be 
\begin{itemize}
\item[{\rm (i)}] {\it firmly nonexpansive} if 
\[
\Vert U(x)-U(y)\Vert^{2}\le \pair{x-y}{U(x)-U(y)} ~~(x,y\in C);
\]
\item[{\rm (ii)}] {\it nonexpansive} if 
\[
 \Vert U(x)-U(y)\Vert\le \Vert x-y\Vert ~~(x,y\in C);
\]
\item[{\rm (iii)}] {\it $\alpha$-averaged} for $\alpha\in (0,1)$, if there exists a nonexpansive mapping $R : C\rightarrow H$ such that
\[
 U=(1-\alpha) I+\alpha R.
\]
\end{itemize}

The {\it metric projection} of a point $x\in H$ onto $C$, denoted by $P_{C}(x)$, is defined as 
a unique solution to problem
\[
\mbox{minimize}~~\Vert x-y\Vert~~\mbox{subject to}~~y\in C.
\]
We know that $P_{C}$ is (firmly) nonexpansive and  
 satisfies $P_{x+C}(y)=P_{C}(y-x)+x$  for all $x,y\in H$. 
See \cite{Bauschke-Combettes}, 
\cite{Deutsch}, \cite{Takahashi} and \cite{Escalante-Raydan} for further information on 
metric projections.  The {\it normal cone} to $C$ at $x$ is defined by 
\[
 N_{C}(x)=\{v\in H : \pair{v}{y-x}\le 0~~\mbox{for all}~~y\in C\}. 
\]

Let $A$ and $B$ be nonempty, closed and convex subsets of $H$. Given $\alpha,\beta\in (0,1)$, we 
define the {\it averaged alternating modified reflections} (AAMR) operator  $T_{A,B,\alpha,\beta} : H\rightarrow H$ 
as 
\[
T_{A,B,\alpha,\beta}=(1-\alpha)I+\alpha(2\beta P_{B}-I)(2\beta P_{A}-I). 
\]
Assume that $R_{A,B,\beta}=(2\beta P_{B}-I)(2\beta P_{A}-I)$. 
We list the following useful properties of $T_{A,B,\alpha,\beta}$ and $R_{A,B,\beta}$:
\begin{itemize}
\item[{\rm (1)}] $(2\beta P_{A}-I)$ (resp. $(2\beta P_{B}-I)$) is nonexpansive and $T_{A,B,\alpha,\beta}$ is $\alpha$-averaged; 
\item[{\rm (2)}] For any $q\in H$,
\begin{enumerate}
\item[{\rm (a)}] ${\rm Fix} (T_{A-q,B-q,\alpha,\beta})= {\rm Fix}(R_{A-q,B-q,\beta})$;
\item[{\rm (b)}] ${\rm Fix}(T_{A-q,B-q,\alpha,\beta})\neq\emptyset$ if and only if  $A\cap B\neq\emptyset$ and 
$q$ satisfies (\ref{CQ1}).
\end{enumerate}
\end{itemize}
See \cite[Sections 3 and 4]{Aragon-Campoy} for more details. 

Let $C$ and $D$ be two closed and convex subsets of $H$. 
The condition (\ref{CQ1})  is important to 
guarantee the existence of fixed points of $T_{A-q,B-q,\alpha,\beta}$. 
The following notion is closely related to (\ref{CQ1}). 
The pair of sets $\{C,D\}$ is said to have the 
{\it strong 
conical hull intersection property} (strong CHIP) at $x\in C\cap D$ if
$N_{C\cap D}(x) = N_{C}(x) + N_{D}(x)$. We say $\{C,D\}$ has the strong CHIP 
if it has the strong CHIP at each $x\in C\cap D$. 
In particular, it was shown in \cite[Proposition 4.1]{Aragon-Campoy} that, for all $q\in H$, 
$q$ satisfies (\ref{CQ1}) if and only if $\{A,B\}$ has the strong CHIP. 
A well-known sufficient condition for the strong CHIP 
is the following standard interior-point-like condition, $A\cap {\rm int}B\neq\emptyset$. 
For more general sufficient conditions for the strong CHIP, see \cite{Deutsch,Burachik-Jeyakumar}. 
The condition $A\cap{\rm int}B\neq\emptyset$ and the following result will be useful in Section 5.
\begin{lemma}
\label{Lem2}
Let $A$  and  $B$ be nonempty sets in $H$. 
If $A\cap {\rm int}B\neq\emptyset$, then for any $e\in H$, there 
exists $\gamma>0$ such that $A\cap {\rm int}(B+\gamma e)\neq\emptyset$. 
\end{lemma}
\begin{proof}
 Let $u\in A\cap {\rm int}B$. Then, there exists $r> 0$ such that 
$B(u,r)\subset B$. 
We can choose sufficiently small $\gamma>0$ 
to make the following holds;
\[
 \Vert u-(u-\gamma e)\Vert=\gamma\Vert e\Vert\le r.
\]
This implies that $u-\gamma e\in B(u,r)\subset B$ and hence 
$u-\gamma e\in {\rm int}B$. 
Since $({\rm int}B+\gamma e)\subset {\rm int}(B+\gamma e)$ (see, e.g., \cite{Tanaka-Kuroiwa}), 
we can therefore conclude that 
$u\in A\cap({\rm int}B+\gamma e)\subset A\cap {\rm int}(B+\gamma e)$.
\end{proof}

\section{Weak convergence result}

This section shows the weak convergence of 
the modification of the AAMR method. 

We consider the following iterative method. 
Choose $x_{0}, q\in H$ and  $\alpha,\beta\in (0,1)$ and  
consider the iterative scheme
\begin{equation}
\label{AAMR3}
\left \{
\begin{array}{l}
y_{n}=P_{A}(x_{n}+q)\\
x_{n+1}=T_{A-q,B-q,\alpha,\beta}(x_{n}),~~n=0,1,2,\dots.
\end{array}
\right.
\end{equation}

Before we proceed with the convergence analysis of (\ref{AAMR3}), 
we introduce the following result.  

\begin{proposition}{\rm \cite[Theorem 2.10]{Bauschke}}
\label{Prop1}
 Set $I = \{1, 2,\dots, m\}$, where $m$ is an integer 
greater than or equal to $2$. Let $\{F_{i}\}_{i\in I}$ be a 
family of firmly nonexpansive mappings on $H$, 
and let, for each $i\in I$, $\{z_{i,n}\}$ be a sequence in $H$ 
such that for all $i, j\in I$, 
\begin{align*}
&z_{i,n}\rightharpoonup z_{i}~ \mbox{and}~F_{i}z_{i,n}\rightharpoonup  x,\\
&\sum_{i\in I}(z_{i,n}-F_{i}(z_{i,n}))\rightarrow -mx+\sum_{i\in I}z_{i},\\
&F_{i}(z_{i,n})-F_{j}(z_{j,n})\rightarrow0. 
\end{align*}
Then $F_{i}(z_{i})=x,$ for every $i\in I$. 
\end{proposition}

The first main result is stated as follows.
\begin{theorem}
\label{Main1}
Let $A$ and $B$ be closed and convex sets in $H$ 
and let $\{y_{n}\}$ be the sequence  generated by (\ref{AAMR3}). 
If $A\cap B\neq\emptyset$ and $q-P_{A\cap B}(q)\in (N_{A}+N_{B})(P_{A\cap B}(q))$,
then $\{y_{n}\}$ weakly converges to $P_{A\cap B}(q)$. 
\end{theorem}
\begin{proof}
Using \cite[Remark 3.2 and Corollary 4.1]{Aragon-Campoy}, we have ${\rm Fix}(R_{A-q,B-q,\beta})\neq\emptyset$. 
Let $u\in {\rm Fix}(R_{A-q,B-q,\beta})$. 
Since $\{x_{n}\}$ can be viewed as the Krasnosel'ski\u\i-Mann 
fixed point iteration with respect to nonexpansive mapping 
$R_{A-q,B-q,\beta}$, by virtue of \cite[Theorem 5.14]{Bauschke-Combettes}, 
we have that, for any $n\in \mathbb{N}$, 
\begin{equation}
\label{Eq2}
\alpha(1-\alpha)\Vert (I-R_{A-q,B-q,\beta})(x_{n})\Vert^{2}\le \Vert x_{n}-u\Vert^{2}
-\Vert x_{n+1}-u\Vert^{2}
\end{equation}
and 
\begin{equation}
\label{Eq3}
\Vert  (I-R_{A-q,B-q,\beta})(x_{n+1})\Vert^{2}\le \Vert  (I-R_{A-q,B-q,\beta})(x_{n})\Vert^{2}. 
\end{equation}
Moreover, 
\begin{equation}
\label{Eq4}
 \Vert (I-R_{A-q,B-q,\beta})(x_{n})\Vert\rightarrow 0~~(n\rightarrow \infty). 
\end{equation}

By \cite[Theorem 4.1]{Aragon-Campoy}, 
\begin{equation}
\label{Eq5}
x_{n}\rightharpoonup x^{*}~~(n\rightarrow\infty), 
\end{equation}
such that  $P_{A}(x^{*}+q)=P_{A\cap B}(q)$. 
Since $P_{A-q}$ is firmly nonexpansive and $\{x_{n}\}$ is bounded, $\{P_{A-q}(x_{n})\}$ 
is bounded. 
Then, there exists a subsequence $\{P_{A-q}(x_{n_{k}})\}$ 
of $\{P_{A-q}(x_{n})\}$ such that $\{P_{A-q}(x_{n_k})\}$ 
weakly converges to some $x\in H$ and hence 
\begin{equation}
\label{Eq6}
P_{A-q}(x_{n_k})\rightharpoonup x~(k\rightarrow\infty). 
\end{equation}

To simplify the notation, define
\[
w_{n}=2\beta P_{A-q}(x_{n})-x_{n},~~ n=0,1,2,\dots.  
\]
From the definition of $R_{A-q,B-q,\beta}$, we have
\begin{align*}
I-R_{A-q,B-q,\beta}&=I-(2\beta P_{B-q}-I)(2\beta P_{A-q}-I)\\
&=I-2\beta P_{B-q}(2\beta P_{A-q}-I)+2\beta P_{A-q}-I\\
&=2\beta(P_{A-q}-P_{B-q}(2\beta P_{A-q}-I)). 
\end{align*}
This together with (\ref{Eq4}) yields
\[
 2\beta \Vert P_{A-q}(x_{n})-P_{B-q}(w_{n})\Vert\rightarrow 0~(n\rightarrow\infty),
\]
and hence 
\begin{equation}
\label{Eq7}
  \Vert P_{A-q}(x_{n})-P_{B-q}(w_{n})\Vert\rightarrow 0~(n\rightarrow\infty). 
\end{equation}
This implies that $\{P_{B-q}(w_{n_k})\}$ weakly converges to $x$ and hence 
\begin{equation}
\label{Eq8}
P_{B-q}(w_{n_k})\rightharpoonup x~(k\rightarrow\infty). 
\end{equation}
Using (\ref{Eq6}) and (\ref{Eq8}), we have 
\begin{equation}
\label{Eq9}
 w_{n_k}\rightharpoonup 2\beta x-x^{*}~~(k\rightarrow\infty), 
\end{equation}
and set $w^{*}=2\beta x-x^{*}$. Using (\ref{Eq5}), (\ref{Eq6}), (\ref{Eq8}) and 
(\ref{Eq9}), we have
\begin{equation}
\label{Eq10} 
x_{n_{k}}-P_{A-q}(x_{n_{k}})+w_{n_{k}}-P_{B-q}(w_{n_{k}})\rightharpoonup -2x+x^{*}+w^{*}~(k\rightarrow\infty). 
\end{equation}
Therefore, the assumptions of Proposition \ref{Prop1} 
are satisfied at this theorem by taking 
\[
 z_{1,k}=x_{n_{k}}, F_{1}(z_{1,k})=P_{A-q}(x_{n_{k}}),
 z_{2,k}=w_{n_{k}}, F_{2}(z_{2,k})=P_{B-q}(w_{n_{k}}),
\]
and we have that 
\[
 P_{A-q}(x^{*})=x.
\]
Since $x$ is an arbitrary weak cluster point of $\{P_{A-q}(x_{n})\}$, 
we conclude that 
\[
P_{A-q}(x_{n})\rightharpoonup  P_{A-q}(x^{*})~(n\rightarrow\infty).
\]
This together with the property of $P_{A}$ yields
\[
P_{A}(x_{n}+q)\rightharpoonup  P_{A}(x^{*}+q)~(n\rightarrow\infty). 
\]
\end{proof}

\begin{remark}
 Since $P_{A}(x^{*}+q)=P_{A\cap B}(q)$ (see \cite[Proposition 3.4]{Aragon-Campoy}), 
(\ref{AAMR3}) generates a sequence weakly converging to the unique solution to
the best approximation problem (\ref{BAP}). 
That is, (\ref{AAMR3}) can directly be applied to solve problem (\ref{BAP}). 
Moreover, we can also show that 
\[
 P_{B}(2\beta P_{A}(x_{n}+q)-x_{n})\rightharpoonup 
P_{A}(x^{*}+q)~(n\rightarrow\infty). 
\]
The proof is much the same as that of Theorem \ref{Main1}.
\end{remark}

\begin{remark}
When $H$ is finite-dimensional, 
$\{x_{n}\}$ strongly converges, and hence 
$\{y_{n}\}$ strongly converges to $P_{A}(x^{*}+q)$. 
Numerical results of (\ref{AAMR3}) were presented 
in \cite[Section 7]{Aragon-Campoy} to demonstrate the 
efficiency in comparison with existing algorithms. 
However, in infinite-dimensional Hilbert space, 
the weak convergence of $\{y_{n}\}$ was not guaranteed because 
$P_{A}$ may fail to be sequentially  weakly continuous \cite{Bauschke,Zarantonello}.
We showed weak convergence of $\{y_{n}\}$, without any other restrictions.
\end{remark}

\section{Convergence rate result}
We next establish the convergence rate of (\ref{AAMR3}). 
To estimate the convergence rate, 
we consider the following residual function 
\begin{equation}
\label{Estimate2}
r(x)=\Vert P_{A}(x+q)-P_{B}(P_{A}(x+q))\Vert.
\end{equation}
Let $\{x_{n}\}$ be 
a sequence generated by (\ref{AAMR1}). Then, from the definition of (\ref{Estimate2}), 
$r$ has the following properties:
\begin{itemize}
\item $r(x)\ge 0~(x\in H)$;
 \item $r(x_{n})=\Vert y_{n}-P_{B}(y_{n})\Vert$;
\item $r(x)=0$ if and only if $P_{A}(x+q)=P_{B}(P_{A}(x+q))\in A\cap B$.
\end{itemize}

The next lemma is useful to our proof of the convergence rate theorem.

\begin{lemma} {\rm \cite[Lemma 1.2]{Deng-Lai-Peng-Yin}}
\label{Lem1}
Let $\{\alpha_{n}\}$ be the sequence in $\mathbb{R}$ such that 
\begin{itemize}
 \item[{\rm (1)}] $\alpha_{n}\ge 0$;
 \item[{\rm (2)}] $\sum_{i=0}^{\infty}\alpha_{i}<\infty$;
 \item[{\rm (3)}] $\{\alpha_{n}\}$ is monotonically non-increasing, 
\end{itemize}
then $\alpha_{n}=o\left(\frac{1}{n}\right)$, where the notation 
$o$ means that $\alpha_{n}=o\left(\frac{1}{n}\right)$ if and only if 
$\lim_{n\rightarrow\infty}\alpha_{n}\cdot n=0$.  
\end{lemma}

The second main result is stated as follows.

\begin{theorem}
 \label{Main2}
Let $A$ and $B$ be closed and convex sets in $H$ 
and let $\{y_{n}\}$ be the sequence  generated by (\ref{AAMR3}). 
If $A\cap B\neq\emptyset$ and $q-P_{A\cap B}(q)\in (N_{A}+N_{B})(P_{A\cap B}(q))$,
then $r(x_{n})=o\left(\frac{1}{\sqrt{n}}\right)$. 
\end{theorem}
\begin{proof}
Let $u\in {\rm Fix}(R_{A-q,B-q,\beta})$. 
By (\ref{Eq2}) in the proof of Theorem \ref{Main1}, we have, for any $n\in \mathbb{N}$, 
\[
 \alpha(1-\alpha)\Vert (I-R_{A-q,B-q,\beta})(x_{n})\Vert^{2}\le \Vert x_{n}-u\Vert^{2}
-\Vert x_{n+1}-u\Vert^{2}. 
\]
Summing up from $j = 0$ to $k$, 
\[
\alpha(1-\alpha) \sum_{j=0}^{k}\Vert (I-R_{A-q,B-q,\beta})(x_{j})\Vert^{2}
\le \Vert x_{0}-u\Vert^{2}-\Vert x_{k+1}-u\Vert^{2}\le \Vert x_{0}-u\Vert^{2},
\]
and hence 
\[
\sum_{j=0}^{\infty}\Vert (I-R_{A-q,B-q,\beta})(x_{j})\Vert^{2}<\infty. 
\]
Obviously, $\Vert (I-R_{A-q,B-q,\beta})(x_{n})\Vert^{2}\ge 0$, 
using the above result and (\ref{Eq3}), the assumptions of Lemma \ref{Lem1} 
are satisfied at this theorem by taking
\[
 \alpha_{n}=\Vert (I-R_{A-q,B-q,\beta})(x_{n})\Vert^{2},
\]
and hence 
\[
 \Vert (I-R_{A-q,B-q,\beta})(x_{n})\Vert^{2}=o\left(\frac{1}{n}\right).   
\]
This implies that 
\[
 n\Vert (I-R_{A-q,B-q,\beta})(x_{n})\Vert^{2}\rightarrow 0~(n\rightarrow \infty),   
\]
and hence 
\begin{equation}
\label{Eq11}
 \sqrt{n}\Vert (I-R_{A-q,B-q,\beta})(x_{n})\Vert\rightarrow 0~(n\rightarrow \infty).
\end{equation}
Using $I-R_{A-q,B-q,\beta}=2\beta(P_{A-q}-P_{B-q}(2\beta P_{A-q}-I))$ and 
the property of the metric projection, we have 
\begin{align*}
\Vert (I-R_{A-q,B-q,\beta})(x_{n})\Vert&=
2\beta\Vert (P_{A-q}-P_{B-q}(2\beta P_{A-q}-I))(x_{n})\Vert \\
&=2\beta\Vert P_{A-q}(x_{n})-P_{B-q}(2\beta P_{A-q}(x_{n})-x_{n})\Vert \\
&=2\beta\Vert P_{A}(x_{n}+q)-q-P_{B}(2\beta P_{A}(x_{n}+q)-q-x_{n}+q)+q\Vert \\
&=2\beta\Vert y_{n}-P_{B}(2\beta y_{n}-x_{n})\Vert.
\end{align*}
This together with (\ref{Eq11}) implies that
\[
2\beta \sqrt{n}\Vert y_{n}-P_{B}(2\beta y_{n}-x_{n})\Vert\rightarrow 0~(n\rightarrow \infty). 
\]
By the definition of $P_{B}$, we have 
\[
\Vert y_{n}-P_{B}(y_{n})\Vert\le  \Vert y_{n}-P_{B}(2\beta y_{n}-x_{n})\Vert
\]
and hence 
\[
 2\beta \sqrt{n}\Vert y_{n}-P_{B}(y_{n})\Vert\rightarrow 0~(n\rightarrow \infty). 
\]
We can therefore conclude that 
\[
 r(x_{n})=o\left(\frac{1}{\sqrt{n}}\right). 
\]
\end{proof}

\begin{remark}
The worst-case convergence rates of 
the Krasnosel'ski\u\i-Mann iterations have been analyzed in 
\cite{C-S-V,L-F-P}. 
We estimated that $r(x_{n})$ converges to zero at a rate of 
$o\left(\frac{1}{\sqrt{n}}\right)$. 
On the other hand, 
it is not guaranteed whether the weak cluster points of the 
Krasnosel'ski\u\i-Mann iterations solve the best 
approximation problem. 
We showed that (\ref{AAMR3}) generates a sequence 
weakly converging to the solution to problem (\ref{BAP}). 

\end{remark}

\section{Finite termination result}
In this section, we investigate finite termination of a modification of (\ref{AAMR3}). 
We make the following assumptions.
\begin{assumption}~
\label{FT}
\begin{itemize}
\item[{\rm (A1)}] $B$ is closed and convex cone;
\item[{\rm (A2)}] $A \cap{\rm int }B\neq\emptyset$.
\end{itemize}
\end{assumption}
\begin{remark}
Assumption (A2) implies that ${\rm int}B\neq\emptyset$. 
Using Lemma \ref{Lem2}, for any $e\in {\rm int}B$, 
$A\cap (B+\gamma e)\neq\emptyset$ for sufficiently small $\gamma>0$. 
\end{remark}

We know the following lemma, due to 
Rami, Helmke and Moore \cite{Rami-Helmke-Moore}. 

\begin{lemma}{\rm \cite[Lemma 2.3]{Rami-Helmke-Moore}}
\label{Lem4}
Let $C$ be a closed and convex cone in $H$ such that $\mbox{\rm int}C\neq\emptyset$. 
If $e\in \mbox{\rm int}C$, then it holds
\begin{equation}
\label{R-H-M}
 \mbox{\rm dist}(C+e,(\mbox{\rm int}C)^{c})>0.
\end{equation}
\end{lemma}

\begin{remark} An example of $C$ satisfying (\ref{R-H-M}) is $\mathcal{S}^{N}_{+}$. 
\begin{itemize}
\item 
 Since  $\mathcal{S}^{N}_{+}$ is a closed and convex cone and 
$\delta I_{N}\in {\rm int}\mathcal{S}^{N}_{+}(=\mathcal{S}^{N}_{++}$) for $\delta>0$, 
${\rm dist}(\mathcal{S}^{N}_{+}+\delta I_{N},(\mathcal{S}^{N}_{++})^{c})>0$, 
where $I_{N}$ is the $N\times N$ identity matrix and 
$\mathcal{S}^{N}_{++}$ is the set of $N\times N$ symmetric positive definite matrices. 
\item For any $\delta>0$, the lower bound of $\mbox{\rm dist}(\mathcal{S}^{N}_{+}+\delta I_{N},(\mathcal{S}^{N}_{++})^{c})$ 
can be estimated by $\delta$, i.e., 
\[
 \mbox{\rm dist}(\mathcal{S}^{N}_{+}+\delta I_{N},(\mathcal{S}^{N}_{++})^{c})\ge \delta
\]
(see \cite[Section 4]{M-X}). 

\end{itemize}
\end{remark}


Suppose that Assumption \ref{FT}. Let $e\in {\rm int}B$ and $\gamma>0$ such that $A\cap (B+\gamma e)\neq\emptyset$.  
The existence of $e$ and $\gamma$ are guaranteed by (A2) and Lemma \ref{Lem2}. 
We consider the following modification of (\ref{AAMR3}). 
Choose $z_{0}$ and  $\alpha,\beta\in (0,1)$ and  
consider the iterative scheme
\begin{equation}
\label{AAMR4}
\left \{
\begin{array}{l}
w_{n}=P_{A}(z_{n})\\
z_{n+1}=T_{A,B+\gamma e,\alpha,\beta}(z_{n}),~~n=0,1,2,\dots.
\end{array}
\right.
\end{equation}

\begin{remark}
In theorems \ref{Main1} and \ref{Main2}, 
we used the metric projections onto the sets shifted by $-p$ 
satisfying (\ref{CQ1}). 
The condition (\ref{CQ1}) is automatically satisfied when $A\cap{\rm int}B\neq\emptyset$ holds (see \cite{Aragon-Campoy}). 
\end{remark}

The third main result is stated as follows.

\begin{theorem}
\label{Main3}
Suppose that  Assumption \ref{FT} holds. Let  $\{w_{n}\}$ 
be the sequence generated by (\ref{AAMR4}), where $e\in {\rm int}B$ and $\gamma >0$ such that $A\cap(B+\gamma e)\neq\emptyset$. 
Then $\{w_{n}\}$ terminates finitely to some point 
$w\in  A\cap {\rm int}B$.
\end{theorem}
\begin{proof} 
Using the assumption (A2) and Lemma \ref{Lem2}, 
$\{A,B+\gamma e\}$ has the strong CHIP (see, e.g., \cite{Burachik-Jeyakumar,Aragon-Campoy}). 
Using \cite[Remark 3.2 and Theorem 3.1]{Aragon-Campoy}, we have ${\rm Fix}(R_{A,B+\gamma e,\beta})\neq\emptyset$. 
Let $u\in {\rm Fix}(R_{A,B+\gamma e,\beta})$. 
Since $\{z_{n}\}$ can be viewed as the Krasnosel'ski\u\i-Mann 
fixed point iteration with respect to nonexpansive mapping 
$R_{A,B+\gamma e,\beta}$, by virtue of \cite[Theorem 5.14]{Bauschke-Combettes}, 
we have that, for any $n\in \mathbb{N}$, 
\begin{equation}
\label{Eq12}
\alpha(1-\alpha)\Vert (I-R_{A,B+\gamma e,\beta})(z_{n})\Vert^{2}\le \Vert z_{n}-u\Vert^{2}
-\Vert z_{n+1}-u\Vert^{2}.
\end{equation}
By summing up (\ref{Eq12}) from $j=0$ to $k$, 
\[
\alpha(1-\alpha) \sum_{j=0}^{k}\Vert (I-R_{A,B+\gamma e,\beta})(z_{j})\Vert^{2}
\le \Vert z_{0}-u\Vert^{2}-\Vert z_{k+1}-u\Vert^{2}\le \Vert z_{0}-u\Vert^{2}.
\]
Using $I-R_{A,B+\gamma e,\beta}=2\beta(P_{A}-P_{B+\gamma e}(2\beta P_{A}-I))$ and 
the  similar arguments as in the proof of Theorem \ref{Main2}, we can show that 
\begin{equation}
\label{Eq13}
 \Vert (P_{A}-P_{B+\gamma e}(2\beta P_{A}-I))(z_{n})\Vert=o\left(\frac{1}{\sqrt{n}}\right). 
\end{equation}

On the other hand, using Lemma \ref{Lem4}, we have
\begin{align*}
{\rm dist}(B+\gamma e, A\cap ({\rm int }B)^{c})\ge {\rm dist}(B+\gamma e, ({\rm int }B)^{c})
>0.  
\end{align*}
Using (\ref{Eq13}), there exists $l_{0}\in \mathbb{N}$ such that 
\begin{equation}
\label{Eq14}
\Vert  (P_{A}-P_{B+\gamma e}(2\beta P_{A}-I))(z_{l})\Vert
< {\rm dist}(B+\gamma e, A\cap ({\rm int }B)^{c})
\end{equation}
for all $l\ge l_{0}$. 
Let $l\in \mathbb{N}$ with $l\ge l_{0}$. 
If  $w_{l}=P_{A}(z_{l})\notin \mbox{int}B$, then 
$w_{l}\in A\cap (\mbox{int}B)^{c}$.
From the definition of $\mbox{dist}(B+\gamma e,A\cap (\rm{int}B)^{c})$, 
we can see that 
\[
 \Vert P_{A}(z_{l})-P_{B+\gamma e}(2\beta P_{A}-I)(z_{l})\Vert\ge 
\mbox{\rm dist}(B+\gamma e,A\cap(\mbox{\rm int}B)^{c}), 
\]
and this is a contradiction to (\ref{Eq14}). Therefore, $w_{l}\in A\cap \mbox{int}B$ 
for all $l\ge l_{0}$. 
\end{proof}

\begin{remark}
Finite termination  of  projection-type iterative methods 
was established in \cite{Rami-Helmke-Moore,B-D-N-P,M-X}. The techniques used in Theorem \ref{Main3} 
can be found in \cite{Rami-Helmke-Moore,M-X}. 
\end{remark}

\section{Conclusion}
In this paper, we have studied variants of the AAMR method for solving 
the best approximation problem in an infinite-dimensional Hilbert space.
In particular, its theoretical properties such as global weak convergence, an 
$o\left(\frac{1}{\sqrt{n}}\right)$ rate and finite termination  are established. 
Our variant has a few advantages. First, the method 
can directly be applied to solve the best approximation problem. 
Second, it guarantees a convergence rate of $o\left(\frac{1}{\sqrt{n}}\right)$.

Although no numerical results are given here, the behavior of (\ref{AAMR3}) can 
be estimated from the computational experience reported in \cite[Section 7]{Aragon-Campoy}, 
since the method in \cite[Section 7]{Aragon-Campoy} is essentially the same as 
that considered in this paper in the finite-dimensional setting.  

\section*{Acknowledgments}
The author is grateful to Professors W. Takahashi of Tokyo Institute of Technology,  
D. Kuroiwa of Shimane University and Li Xu of Akita Prefectural University for their helpful support.

\end{document}